\documentclass[50pt]{article}
\usepackage{mathtools}
\usepackage{amsmath}
\usepackage[latin1]{inputenc}
\usepackage{amsmath,amsthm,amssymb}
\usepackage{graphicx}
\usepackage{color}

\renewcommand{\theequation}{\thesection.\arabic{equation}}
\newtheorem{df}{Definition}[section]
\newtheorem{lm}{Lemma}[section]
\newtheorem{thm}{Theorem} [section]

\newtheorem{cor}{Corollary}[section]
\newtheorem{rem}{Remark}[section]
\title{Well-posedness and  persistence properties  for two-component higher order Camassa-Holm systems with fractional inertia operator
\thanks{This work is supported by National Science Fund for Young Scholars of China (Grant No. 11301573), University Young Core  Teacher Foundation of Chongqing, Technology Research Foundation of Chongqing Educational Committee (Grant No. KJ1400503), Natural Science Foundation of Chongqing  (Grant No. cstc2014jcyjA00008),   the Talent Project of Chongqing Normal University(Grant No. 14CSBJ05). The third author is supported by National Science Foundation of China (Grant No.11371384) and Natural Science Foundation of Chongqing (Grant N0. cstc2015jcyjBX0007).}}
\author{Rong Chen and Shouming Zhou\footnote{E-mail:
zhoushouming76@163.com } \\
\small   College of Mathematics Science, Chongqing Normal University, Chongqing 401331, China.}

\begin{document}
\maketitle
\renewcommand{\theequation}{\arabic{section}.\arabic{equation}}
\catcode`@=11 \@addtoreset{equation}{section} \catcode`@=12
\textbf{Abstract.}
In this paper, we study the Cauchy problem for a two-component higher order Camassa-Holm systems with fractional inertia operator $A=(1-\partial_x^2)^r,r\geq1$, which was proposed by Escher and Lyons \cite{EL}.
By the transport equation theory and Littlewood-Paley decomposition,  we obtain that the local well-posedness of solutions for the system
in  nonhomogeneous Besov spaces $B^s_{p,q}\times B^{s-2r+1}_{p,q}$ with $1\leq p,q \leq +\infty$ and the Besov index   $s>\max\left\{2r +\frac{1}{p},2r+1-\frac{1}{p}\right\}$.
Moreover, we construct the local well-posedness in the critical Besov space $B^{2r+\frac{1}{2}}_{2,1}\times B^{\frac{3}{2}}_{2,1}$.  On the other hand,   the propagation behaviour of compactly supported solutions is examined, namely whether solutions which
are initially compactly supported will retain this property throughout their time of evolution.
Moreover, we also establish the
persistence properties  of the solutions  to the two-component Camassa-Holm
equation with $r=1$ in weighted $L_{ \phi}^p:=L^p(\mathbb{R},\phi^p(x)dx)$   spaces  for a large class of moderate weights. \\
{\bf Keywords}: higher order Camassa-Holm systems, well-posedness, Besov spaces, infinite propagation speed,  persistence properties.\\
{\bf Mathematics Subject Classification(2000)}: 35G25, 35L05, 35Q50, 35Q53, 37K10.

\section{Introduction }
 In this paper, we consider the following Cauchy problem
\begin{equation}\label{Eq.(1.1)}
\left\{
\begin{array}{llll}
m_{t}=\alpha   u_x-bu_xm-um_x-\kappa \rho \rho_x, \,\ m=Au,&t>0,x\in\mathbb{R} , \\
\rho_t=-u \rho_x-(b-1)u_x\rho ,\,\ \alpha_t=0 &t>0,x\in\mathbb{R},\\
u(x,0)=u_0(x),\,\ \rho(x,0)=\rho_0(x),&t=0,x\in\mathbb{R},
\end{array}
\right.
\end{equation}
where the inertia operator $A=(1-\partial_x^2)^r$ belongs to the class of fractional Sobolev norms $r\geq1$, and the constants $ b\in\mathbb{R}  ,\kappa\in\mathbb{R}$.

Obviously, if $\rho\equiv0,A=1-\partial_x^2$ and  $b\in\mathbb{R},\alpha_x=0$ the Eq.(\ref{Eq.(1.1)}) becomes a one-component family of equations which are parameterised by $b $:
\begin{equation}\label{Eq.(1.2)}
u_t-u_{xxt}+c_0u_x+(b+1)uu_x=bu_xu_{xx}+uu_{xxx}.
\end{equation}
 This
family of so-called $b$-equations possess a number of structural phenomena which are shared
by solutions of the family of equations (c.f.\cite{EY,ZhouMu,ZMW}).
 By using
Painlev\'{e} analysis, there are  only
two asymptotically integrable within this family:   the Camassa-Holm (CH) equation (Eq.(\ref{Eq.(1.2)}) with $b=2$, c.f. \cite{CH,CHH})  and the Degasperis-Procesi (DP) equation (Eq.(\ref{Eq.(1.2)}) with $b=3$, c.f. \cite{DP}). Integrable equations have widely been studied because they usually have very good properties including  infinitely many conservation laws, infinite higher-order symmetries, bi-Hamiltonian
structure, and Lax pair, which make them solved by the inverse scattering method. Conserved quantities are very feasible for proving the existence of global solution in time, while a bi-Hamiltonian
formulation helps in finding conserved quantities effectively. The advantage of the CH and DP equation in
comparison with the KdV equation lie  in the fact that they not only have peaked solitons but also model the
peculiar wave breaking phenomena (c.f. \cite{CHH,CE2}), and hence they represent the first examples of
integrable equations which possess both global solutions
and solutions which display wave-breaking in finite time, c.f. \cite{C,CE,CE2}.

In recent years, a number of integrable multi-component generalisations
of the CH equation has been studied extensively. One of them is the following family of two-component systems parameterised by $b$
\begin{equation}
\left\{
\begin{array}{llll}
m_{t}=\alpha   u_x-bu_xm-um_x-\kappa \rho \rho_x,\,\  m=u-u_{xx},\\
\rho_t=-u \rho_x-(b-1)u_x\rho ,\,\  \alpha,b\in\mathbb{R},
\end{array}
\right.
\end{equation}
Apparently,  the two-component Camassa-Holm system  \cite{CI}, and
 the two-component Degasperis-Procesi system \cite{P} are included in Eq. (1.3) as two special cases with $b=2 $ and $b=3$, respectively.  Recently, Escher et al. presented the hydrodynamical derivation
of the system (1.3), as a model for water waves with $\alpha$ a
constant incorporating an underlying vorticity of the flow \cite{EHKL}.  They also proved the local
well-posedness of (1.3) using a geometrical framework, studied the blow-up scenarios and global strong
solutions of (1.3) on the circle.

All of these hydrodynamical models have a geometrical interpretation in terms
of a geodesic flow on an appropriate infinite dimensional Lie group. The seminal
work of Arnold \cite{A} reformulated the Euler equation describing an ideal fluid, as a
geodesic flow on the group of volume preserving diffeomorphisms of the fluid domain.
Following this, Ebin and Marsden \cite{EM} reinterpreted this group of volume preserving
diffeomorphisms as an inverse limit of Hilbert manifolds.  This technique
has been used in \cite{CK,M} for the periodic CH equation. It was extended to (nonmetric) geodesic
flows such as the DP equation in \cite{EK1} and to right-invariant metrics induced by fractional
Sobolev norm (non-local inertia operators) in \cite{EK2}.   In 2009, Mclachlan and Zhang \cite{MZ} studied the
Cauchy problem for a modified CH equation derived as the Euler-Poincar\'{e} differential equation on the
Bott-Virasoro group with respect to the $H^k$ metric, i.e.,
\begin{equation}
m_t+2u_xm+um_x=0,\,\ m= (1-\partial_x^2)^ku,\,\ k\in\mathbb{N}.
\end{equation}
In \cite{CHK}, Coclite, Holden and Karlsen considered higher order Camassa-Holm equations (1.4) describing
exponential curves of the manifold of smooth orientation-preserving diffeomorphisms of the unit
circle in the plane. Recently, in \cite{EL}, Escher and Lyons shown that
the system in equation (1.1)  corresponds to a metric induced geodesic flow on the
infinite dimensional Lie group Diff$^\infty(\mathbb{S}^1) \circledS C^\infty(\mathbb{S}^1) \times \mathbb{R}$,  where
   Diff$ ^\infty(\mathbb{S}^1)  $ denotes the group of orientation preserving diffeomorphisms of
the circle,  $  C^\infty(\mathbb{S}^1)  $ denotes the space of smooth function on  $ \mathbb{S}^1 $ while $ \circledS   $ denotes
an appropriate semi-direct product between the pair.

However, the Cauchy problem of (1.1) on the line has not been studied yet. In this paper, using the
Littlewood-Paley theory, we established the local well-posedness of (1.1) in nonhomogeneous Besov spaces. Furthermore, We examine the propagation
behaviour of compactly supported solutions, namely whether solutions which
are initially compactly supported will retain this property throughout their
time of evolution.   In \cite{GHY},  the authors shown that the solution of Eq.(1.3)
has exponential decay if the initial data in $H^s\times H^{s-1}$ with $s>5/2$ and has exponential decay, in present paper, working with moderate weight functions that are commonly used in time-frequency analysis \cite{Br}, we generalize the persistence result on the solution to Eq. (1.3) in the weighted
$L^p=L^p(\mathbb{R},\phi^p(x)dx)$ spaces, and we also extend the Sobolev index to $s>3/2$. 

Our paper is organized as follows. In Section 2, we establish the local well-posedness of the Cauchy problem associated with (1.1) in
Besov spaces. In Section 3, we discuss the infinite propagation speed for (1.1).   In the last section, we establish
persistence properties and some unique continuation properties of the solutions to the Eq.(1.3)   in weighted  $L^p_{ \phi}:=L^p(\mathbb{R},\phi^p(x)dx)$  spaces. 

\section{Local well-posedness in Besov spaces}
In this section, we shall discuss the local well-posedness of the Cauchy problem
 Eq.(1.1)  in the nonhomogeneous Besov spaces. The Littlewood-Paley theory and the properties
of the Besov-Sobolev spaces can refer to \cite{B,D23} and references therein.

\subsection{Local well-posedness in Besov spaces $B_{p,q}^s$ }
In this section, we shall discuss the local well-posedness of the Cauchy problem (1.1).
First, we present the following definition.
 \begin{df}
For $T>0,s\in\mathbb{R}$ and $1\leq p\leq +\infty$ and $s\neq 2+\frac{1}{p}$, we define
$$E_{p,q}^s(T)\doteq \mathcal{C}([0,T];B_{p,q}^s)\cap \mathcal{C}^1([0,T];B_{p,q}^{s-1})  \quad \text{ if } r<+\infty,$$
$$E_{p,\infty}^s(T)\doteq L^\infty([0,T];B_{p,\infty}^s)\cap Lip([0,T];B_{p,\infty}^{s-1}),$$
and $E_{p,q}^s\doteq \cap_{T>0}E_{p,q}^s(T)$.
\end{df}
The result of the local well-posedness in the Besov space may now be stated.

\begin{thm}\label{result1}
Suppose that $1\leq p,q \leq +\infty$ and   $s>\max\left\{2r +\frac{1}{p}, 2r+1-\frac{1}{p}\right\}$  with $r\geq1$, Let  the function $\alpha(x)\equiv c\in \mathbb{R}$ or $\alpha(\cdot)\in B_{p,q}^{s-2r }$, and the initial data $(u_0,\rho_0)\in B_{p,q}^{s}\times B_{p,q}^{s-2r+1}$.  Then there exists a time $T>0$ such that the Cauchy problem (1.1) has a unique solution $(u,\rho)\in E_{p,q}^{s}(T)\times E_{p,q}^{s-2r+1}(T)$, and map $(u_0,\rho_0)\mapsto (u,\rho)$ is continuous from a neighborhood of $(u_0,\rho_0)$ in $B_{p,q}^{s}\times B_{p,q}^{s-2r+1}$ into
$$\mathcal{C}([0,T];B_{p,q}^{s'}) \cap \mathcal{C}^1([0,T];B_{p,q}^{s'-1})\times \mathcal{C}([0,T];B_{p,q}^{s'-2r+1}) \cap \mathcal{C}^1([0,T];B_{p,q}^{s'-2r})$$
for every $s'<s$ when $r=+\infty$ and $s'=s$ whereas $r<+\infty$.
\end{thm}

\begin{rem}
When $p=q=2$, the Besov space $B^s_{p,q}$
coincides with the Sobolev space $H^s$. Thus under the condition
$(u_0,\rho_0)\in H^s\times H^{s-2r+1}$ with $s>2r+\frac{1}{2}$ the above theorem implies
 that there exists a time $T>0$ such that the initial-value problem (1.1) has a unique
 solution $(u,\rho)\in \mathcal{C}([0,T];H^s)\cap \mathcal{C}^1([0,T];H^{s-1})\times \mathcal{C}([0,T];H^{s-2r+1})\cap \mathcal{C}^1([0,T];H^{s-2r}) $, and the map $(u_0,\rho_0) \mapsto (u,\rho)$ is continuous from a
 neighborhood of $(u_0,\rho_0)$ in $H^{s }\times H^{s-2r+1}$ into $\mathcal{C}([0,T];H^s)\cap \mathcal{C}^1([0,T];H^{s-1})\times \mathcal{C}([0,T];H^{s-2r+1})\cap \mathcal{C}^1([0,T];H^{s-2r})$.
\end{rem}

In the following, we denote $C>0$ a generic constant only depending on $p,q,s$. Uniqueness and
continuity with respect to the initial data are an immediate consequence of the following result.
\begin{lm}
 Let $1\leq p, q\leq +\infty$ and   $s>\max\left\{2r +\frac{1}{p}, 2r+1-\frac{1}{p}\right\}$ with $r\geq1$, the function $\alpha(x)\equiv c\in \mathbb{R}$ or $\alpha(\cdot)\in B_{p,q}^{s-2r }$. Suppose that  $(u_{i}, \rho_{i})\in \{L^\infty([0,T];B_{p,q}^s)\cap \mathcal{C}([0,T];\mathcal{S}')\}\times \{L^\infty([0,T];B_{p,q}^{s-2r+1})\cap \mathcal{C}([0,T];\mathcal{S}')\}$ $(i=1,2)$
be two
given solutions of the initial-value problem (1.1) with the initial data $(u_{i}(0),\rho_{i}(0))\in B_{p,q}^s\times B_{p,q}^{s-2r+1}$  ($i=1,2$), and denote $\rho_{12}=\rho_{1}-\rho_{2},u_{12}=u_{1}-u_{2},$ i.e.,  $m_{12}=m_{1}-m_{2} $. Then for every $t\in[0,T]$, we have

(i) if $s>\max\left\{2r+\frac{1}{p}, 2r+1-\frac{1}{p}\right\} $ and $s\neq 2r+2+\frac{1}{p}$, then
\begin{equation}
\begin{split}
\|u_{12}\|&_{B^{s-1}_{p,q}}+\|\rho_{12}\|_{B^{s-2r}_{p,q}} \leq \left(\| u_{12} (0)\|_{B^{s-1}_{p,q}} +\|\rho_{12} (0)\|_{B^{s-2r}_{p,q}} \right)  \exp
\left(C\int^{t}_{0}\Gamma_{s}(t,\cdot)d\tau\right) .
\end{split}
\end{equation}
where
\begin{equation*}
\Gamma_{s}(t,\cdot)=\left( \|  u_{1}  \|_{B^{s}_{p,q}}  + \|  u_{2}  \|_{B^{s}_{p,q}} +\|\rho_{1}\|_{B^{s-2r+1}_{p,q}}  +\|\rho_{2}\|_{B^{s-2r+1}_{p,q}}  +\|\alpha\|  \right),
\end{equation*}
with
$$\|\alpha\| \doteq\left\{
\begin{array}{llll}
|c|, &\alpha(x)\equiv c\in \mathbb{R} ,\\
 \|\alpha\|_{B^{s-2r  }_{p,q}},&\alpha(\cdot)\in B_{p,q}^{s-2r }.
\end{array}
\right.$$

(ii) If $s= 2r+2+1/p$, then
\begin{equation*}
\begin{split}
\|u_{12}\|&_{B^{s-1}_{p,q}}+\|v_{12}\|_{B^{s-1}_{p,q}} \leq C\left(\| u_{12} (0)\|_{B^{s-1}_{p,q}} +\| u_{12} (0)\|_{B^{s-1}_{p,q}} \right) ^\theta \Gamma_{s}^{1-\theta}(t,\cdot)   \exp
\left(C\theta\int^{T}_{0}\Gamma_{s}(t,\cdot)d\tau\right),
\end{split}
\end{equation*}
where $\theta\in(0,1)$(i.e., $\theta=\frac{1}{2}(1-\frac{1}{2p})$) and $\Gamma_{s}(t,\cdot)$ as in case (i).
\end{lm}

\begin{proof}   It is obvious that
$u_{12}\in L^\infty([0,T];B_{p,q}^s)\cap \mathcal{C}([0,T];\mathcal{S}')$ and $\rho_{12}\in L^\infty([0,T];B_{p,q}^{s-2r+1})\cap \mathcal{C}([0,T];\mathcal{S}')$
which implies that $u_{12}\in \mathcal{C}([0,T];B_{p,q}^{s-1})$, $ \rho_{12}\in \mathcal{C}([0,T];B_{p,q}^{s-2r})$, and $(u_{12},\rho_{12})$   solves the transport equations
\begin{equation*}
\left\{
\begin{array}{llll}
&\partial_{t}m_{12}+ u_{1}  \partial_{x} m_{12}=- u_{12}   \partial_{x}m_{2} -b m_{12} \partial_x u_{1}   -b m_{2} \partial_x u_{12}  -\alpha \partial_xu_{12}-\kappa \partial_x[(\rho_1+\rho_2)\rho_{12}] ,\\
&\partial_{t}\rho_{12}+u_{1} \partial_{x} \rho_{12} =- u_{12} \partial_{x}\rho_{2} -(b-1)\rho_{12}\partial_xu_{1}  -(b-1)\rho_{2}\partial_xu_{12}  ,\\
&u_{12}|_{t=0}=u_{12}(0)\doteq u_{1}(0)-u_{2}(0),\,\ \rho_{12}|_{t=0}=\rho_{12}(0)\doteq \rho_{1}(0)-\rho_{2}(0),\,\ m_{12}=(1-\partial_x^2)^ru_{12}.
\end{array}
\right.
\end{equation*}
\textbf{Case I:   $s>\max\left\{2r+\frac{1}{p}, 2r+1-\frac{1}{p}\right\} $  and $s\neq 2r+2+\frac{1}{p}$.} According to Theorem 3.14 in \cite{B}, we have
\begin{equation}
\begin{split}
 \|m_{12}\|_{B^{s-2r-1}_{p,q}} \leq&  \|m_{12}(0)\|_{B^{s-2r-1}_{p,q}} +C\int_0^t\left(\|\partial_{x} u_{1}  \|_{B^{s-2r-2}_{p,q}} +\|\partial_{x} u_{1} \|_{B^{\frac{1}{p}}_{p,q}\cap L^\infty}\right) \|m_{12}\|_{B^{s-2r-1}_{p,q}}d\tau\\
 &+C\int_0^t \| u_{12}   \partial_{x}m_{2} +b m_{12} \partial_x u_{1}   +b m_{2} \partial_x u_{12}  +\alpha \partial_xu_{12}+\kappa \partial_x[(\rho_1+\rho_2)\rho_{12}]  \|_{B^{s-2r-1}_{p,q}}d\tau.
\end{split}
\end{equation}
Since $s>\max \left\{2r +\frac{1}{p}, 2r+1-\frac{1}{p}\right\}\geq 2r +\frac{1}{p}$ with $r\geq1$, we obtain
\begin{align*}
 \|\partial_{x} u_{1}  \|_{B^{s-2r-2}_{p,q}} +\|\partial_{x} u_{1}  \|_{B^{\frac{1}{p}}_{p,q}\cap L^\infty} \leq 2\|\partial_{x} u_{1}  \|_{B^{s-2r }_{p,q}}\leq C\|u_{1}\|_{B^{s-r}_{p,q}} .
\end{align*}
Note that the inertia operator $(1-\partial_x^2)\in OP (S^2)$, by Proposition 2.2 (7) in \cite{ZMW}, we find that for all $s\in \mathbb{R}$ %that
\begin{equation}
\|u_{i}\|_{B^{s}_{p,q}} \approxeq \|m_{i}\|_{B^{s-2r}_{p,q}} .
\end{equation}
If $\max\left\{ 2r+\frac{1}{p},2r+1-\frac{1}{p} \right\} <s\leq 2r+1+\frac{1}{p}$, by Morse-type estimate  2.5 (2) in \cite{ZMW} with $s_1=s-2r-1,s_2=s-2r$,  it is easy to check that $2r+1\leq\max\left\{ 2r+\frac{1}{p},2r+1-\frac{1}{p} \right\}$, i.e., $s_1+s_2=2s-4r-1>0$. Recall that $B^{s-2r}_{p,q}$
being an algebra for $s-2r\geq1/p$, one can yields
\begin{equation*}
\begin{split}
 \|u_{12}   \partial_{x}m_{2} +&b m_{12} \partial_x u_{1}   +b m_{2} \partial_x u_{12}  +\alpha \partial_xu_{12}+\kappa \partial_x[(\rho_1+\rho_2)\rho_{12}] \|_{B^{s-2r-1}_{p,q}}    \\
&\leq C\bigg( \|  u_{12}  \|_{B^{s-2r }_{p,q}}  \|\partial_{x}m_{ 2} \|_{B^{s-2r-1}_{p,q}} +\|m_{12}\|_{B^{s-2r-1}_{p,q}}\|\partial_xu_{1 }  \|_{B^{s-2r }_{p,q}}+\|m_{ 2}\|_{B^{s-2r}_{p,q}}\|\partial_xu_{12 }  \|_{B^{s-2r -1 }_{p,q}}\\
&\quad+\|\alpha\| \|u_{12}\|_{B^{s-2r  }_{p,q}}+  \|  (\rho_1+\rho_2)\rho_{12}  \|_{B^{s-2r }_{p,q}}\bigg)\\
&\leq C\bigg( \|  u_{12}  \|_{B^{s-2r }_{p,q}}  \|u_{ 2} \|_{B^{s  }_{p,q}} +\|u_{12}\|_{B^{s -1}_{p,q}}\|u_{1 }  \|_{B^{s-2r+1 }_{p,q}}+\|u_{ 2}\|_{B^{s }_{p,q}}\|u_{12 }  \|_{B^{s-2r }_{p,q}}\\
&\quad+\|u_{12}\|_{B^{s-2r }_{p,q}}+  \|  (\rho_1+\rho_2)\rho_{12}  \|_{B^{s-2r }_{p,q}}\bigg)\\
&\leq C \|  u_{12}  \|_{B^{s-1 }_{p,q}}  \bigg(   \|u_{1 }  \|_{B^{s  }_{p,q}}+\|u_{ 2}\|_{B^{s  }_{p,q}}  +\|\alpha\| \bigg)+\left( \|  \rho_1 \|_{B^{s-2r}_{p,q}}+\|\rho_2  \|_{B^{s-2r }_{p,q}}\right)\|\rho_{12}  \|_{B^{s-2r }_{p,q}}.
\end{split}
\end{equation*}
For $s>2r+1+\frac{1}{p}$, the above inequality   also holds true in view of the fact that $B^{s-2r-1}_{p,q}$ is an algebra. Thus, we may derive
\begin{align*}
 \|u_{12}\|_{B^{s-1}_{p,q}} \leq&  \|u_{12}(0)\|_{B^{s-1}_{p,q}} +C\int_0^t \left(\|u_{12} (\tau,\cdot) \|_{B^{s-1 }_{p,q}}+\|\rho_{12} (\tau,\cdot) \|_{B^{s-2r }_{p,q}}\right) \Gamma_s(\tau )d\tau.
\end{align*}
Similarly, we can also treat the inequality for the component   $\rho$:
\begin{equation}
\begin{split}
 \|\rho_{12}\|_{B^{s  -2r }_{p,q}} \leq&  \|\rho_{12}(0)\|_{B^{s-2r }_{p,q}} +C\int_0^t\left(\|\partial_{x} u_{1}  \|_{B^{s-2r -1}_{p,q}} +\|\partial_{x} u_{1} \|_{B^{\frac{1}{p}}_{p,q}\cap L^\infty}\right) \|\rho_{12}\|_{B^{s-2r  }_{p,q}}d\tau\\
 &+C\int_0^t \| u_{12} \partial_{x}\rho_{2} +(b-1)\rho_{12}\partial_xu_{1}  +(b-1)\rho_{2}\partial_xu_{12}  \|_{B^{s-2r   }_{p,q}}d\tau.
\end{split}
\end{equation}
The condition $s -2r-1> \frac{1}{p}$ follows that
\begin{align*}
 \|\partial_{x} u_{1}  \|_{B^{s -2r-1}_{p,q}} +\|\partial_{x} u_{1}  \|_{B^{\frac{1}{p}}_{p,q}\cap L^\infty} \leq 2\|\partial_{x} u_{1}  \|_{B^{s-2r-1 }_{p,q}}\leq C\|u_{1}\|_{B^{s-2r }_{p,q}} .
\end{align*}
If $2r+\frac{1}{p}  <s$, then $B^{s-2r}_{p,q}$
being an algebra, we arrive at
\begin{equation}
\begin{split}
 \| u_{12} \partial_{x}\rho_{2} &+(b-1)\rho_{12}\partial_xu_{1}  +(b-1)\rho_{2}\partial_xu_{12}  \|_{B^{s-2r}_{p,q}} \\
&\leq C\bigg( \|  u_{12}  \|_{B^{s-2r }_{p,q}}  \|\rho_{ 2} \|_{B^{s-2r+1}_{p,q}} +\|\rho_{12}\|_{B^{s-2r }_{p,q}}\|u_{1 }  \|_{B^{s-2r+1 }_{p,q}}+\|\rho_{ 2}\|_{B^{s-2r }_{p,q}}\|u_{12 }  \|_{B^{s-2r+1  }_{p,q}} \bigg)\\
&\leq C \|  u_{12}  \|_{B^{s-1 }_{p,q}}   \|\rho_{ 2} \|_{B^{s -2r+1}_{p,q}}  +  \| u_1 \|_{B^{s-1}_{p,q}}  \|\rho_{12}  \|_{B^{s-2r }_{p,q}}.
\end{split}
\end{equation}
Combining (2.4)-(2.5), we get
\begin{align*}
 \|\rho _{12}\|_{B^{s-2r}_{p,q}} \leq&  \|\rho_{12}(0)\|_{B^{s-2r}_{p,q}} +C\int_0^t  \left(\|u_{12} (\tau,\cdot) \|_{B^{s-1 }_{p,q}}+\|\rho_{12} (\tau,\cdot) \|_{B^{s-2r }_{p,q}}\right) \Gamma_s(\tau )  d\tau.
\end{align*}
Therefore
\begin{align*}
  \|u _{12}\|_{B^{s-1}_{p,q}}  +\|\rho _{12}\|_{B^{s-2r}_{p,q}} \leq &\|u_{12}(0)\|_{B^{s-1}_{p,q}}  + \|\rho_{12}(0)\|_{B^{s-2r}_{p,q}}+C\int_0^t\left(  \|u_{12}  \|_{B^{s-1}_{p,q}} +\|\rho_{12}  \|_{B^{s-2r}_{p,q}}\right)\Gamma_s(\tau,\cdot)d\tau.
\end{align*}
Applying Gronwall's lemma to the above inequality leads to (i).

\textbf{For the critical case (ii) $s=2r+2+1/p$}, we here use the interpolation method to deal with it. Indeed, if we choose $\theta=\frac{1}{2}(1-\frac{1}{2p})\in (0,1)$, then $s-1=2r+1+\frac{1}{p}=\theta (2r+\frac{1}{2p})+(1-\theta)(2r+2+\frac{1}{2p})$ . According to Proposition 2.2(5) in \cite{ZMW} and the above inequality, we have
\begin{align*}
  \|u _{12}\|_{B^{2r+1+1/p}_{p,q}}  +\|\rho _{12}\|_{B^{2r+1+1/p}_{p,q}} &\leq \|u _{12}\|_{B^{2r+\frac{1}{2p}}_{p,q}}^{\theta} \|u _{12}\|_{B^{2r+2+\frac{1}{2p}}_{p,q}}^{1-\theta} +\|\rho_{12}\|_{B^{2r+\frac{1}{2p}}_{p,q}}^{\theta} \|\rho _{12}\|_{B^{2r+2+\frac{1}{2p}}_{p,q}}^{1-\theta} \\
  &\leq \left(\|u _{12}\|_{B^{2r+\frac{1}{2p}}_{p,q}}+\|\rho_{12}\|_{B^{2r+\frac{1}{2p}}_{p,q}}\right)^{\theta} \left(\|u _{12}\|_{B^{2r+2+\frac{1}{2p}}_{p,q}}^{1-\theta}+\|\rho _{12}\|_{B^{2r+2+\frac{1}{2p}}_{p,q}}^{1-\theta}\right)\\
   &\leq C\left(\| u_{12} (0)\|_{B^{3+\frac{1}{p}}_{p,q}} +\| \rho_{12} (0)\|_{B^{3+\frac{1}{p}}_{p,q}} \right) ^\theta \Gamma_{2r+2+\frac{1}{p}}^{1-\theta}(t,\cdot)   \exp
\left(C\theta\int^{T}_{0}\Gamma_{2r+2+\frac{1}{p}}(t,\cdot)d\tau\right),
\end{align*}
which yields the desired result.
\end{proof}
Now, let us start the proof of Theorem 2.1, which is motivated by the proof of local existence theorem about Camassa-Holm type equations  in \cite{D23,ZMW}. We shall use the classical Friedrichs regularization method to construct the approximate solutions to  Eq.(1.1) .
\begin{lm}
Let $p, q,r,s$ and $\alpha$ be as in the statement of Lemma 2.1. Assume that $u(0) =
\rho(0):=0$. There exists a sequence of smooth functions $(u_k,\rho_k)\in \mathcal{C}(\mathbb{R}^+; B_{p,q}^\infty)^2$ solving

$$(T_k) \quad\left\{
\begin{array}{llll}
\partial_{t}m_{k+1}+u_{k }\partial_xm_{k+1}-\alpha u_{k }+b\partial_xu_{k}m_{k}+\kappa \rho_{k}\partial_x\rho_{k}=0, \\
\partial_{t}\rho_{k+1}+u_{k }\partial_x\rho_{k+1}+(b-1)\partial_xu_{k }\rho_{k}=0,\\
u_{k+1}(0)=S_{k+1}u(0),\rho_{k+1}(0)=S_{k+1}\rho(0).
\end{array}
\right.$$

 Moreover, there is a positive time $T$ such that the solutions satisfying the following properties:

(i) $(u_k,\rho_k)_{k\in \mathbb{N}}$ is uniformly bounded in $E_{p,q}^{s}(T)\times E_{p,q}^{s-2r+1}(T)$.

(ii) $(u_k,\rho_k)_{k\in \mathbb{N}}$ is a Cauchy sequence in $\mathcal{C}([0,T];B_{p,q}^{s-1})\times \mathcal{C}([0,T];B_{p,q}^{s-2r })$.
\end{lm}
\begin{proof}
Since all the data $S_{k+1}u_0$ and  $S_{k+1}\rho_0$ belong to $B_{p,q}^\infty$, Lemma 2.3 in \cite{ZMW} indicates %enables us to show by induction that
that for all $k\in\mathbb{N}$, the equation $(T_k)$ has a global solution in %which belongs to
$\mathcal{C}(\mathbb{R}^+;B_{p,q}^\infty)^2$.

For $s>\max\left\{2r+\frac{1}{p}, 2r+1-\frac{1}{p}\right\} $ and $s\neq 2r+2+\frac{1}{p}$, thanks to Lemma 2.2 in \cite{ZMW} and the proof of
Lemma 2.1, we have the following inequality for all $k\in\mathbb{N}$:
\begin{equation*}
\begin{split}
 &\|m_{k+1}\|_{B^{s-2r }_{p,q}}  \leq \exp\left( C\int_0^t\left(\|\partial_{x} u_{k}  (\tau)\|_{B^{s-2r-1}_{p,q}} +\|\partial_{x} u_{k} (\tau)\|_{B^{\frac{1}{p}}_{p,q}\cap L^\infty}\right) d\tau\right)\|m(0)\|_{B^{s-2r}_{p,q}}\\
 &+C\int_0^t\exp\left( C\int_\tau^t\left(\|\partial_{x} u_{k}  (\tau')\|_{B^{s-2r-1}_{p,q}} +\|\partial_{x} u_{k}  (\tau')\|_{B^{\frac{1}{p}}_{p,q}\cap L^\infty}\right) d\tau'\right)\|\alpha u_{k }+b\partial_xu_{k}m_{k}+\kappa \rho_{k}\partial_x\rho_{k}\|_{B^{s-2r }_{p,q}}d\tau.
\end{split}
\end{equation*}
and
\begin{equation*}
\begin{split}
 \|\rho_{k+1}\|_{B^{s-2r +1}_{p,q}} &\leq \exp\left( C\int_0^t\left(\|\partial_{x} u_{k}  (\tau)\|_{B^{s-2r }_{p,q}} +\|\partial_{x} u_{k} (\tau)\|_{B^{\frac{1}{p}}_{p,q}\cap L^\infty}\right) d\tau\right)\|\rho(0)\|_{B^{s-2r+1}_{p,q}}\\
&+C\int_0^t\exp\left( C\int_\tau^t\left(\|\partial_{x} u_{k}  (\tau')\|_{B^{s-2r }_{p,q}} +\|\partial_{x} u_{k}  (\tau')\|_{B^{\frac{1}{p}}_{p,q}\cap L^\infty}\right) d\tau'\right)\|(b-1)\partial_xu_{k }\rho_{k}\|_{B^{s-2r +1}_{p,q}}d\tau.
\end{split}
\end{equation*}
Due to $s>2r+\frac{1}{p}$, we know that $B^{s-2r}_{p,q}$ is an algebra and $B^{s-2r}_{p,q}\hookrightarrow L^\infty$. Furthermore, using the relationship $\|u \|u_{B^{s}_{p,q}} \approxeq \|m \|_{B^{s-2r}_{p,q}} $ implies
\begin{equation*}
\begin{split}
 \|u_{k+1}\|_{B^{s }_{p,q}} &\leq \exp\left(2C\int_0^t \|  u_{k} (\tau)\| _{B^{s }_{p,q}}   d\tau\right)\|u(0)\|_{B^{s }_{p,q}}\\
&+C\int_0^t\exp\left(2C\int_\tau^t \|  u_{k} (\tau')\| _{B^{s }_{p,q}}  d\tau'\right)\left(\|\alpha\|  \|u_{k }\| _{B^{s-2r  }_{p,q}}+\|u_{k }\|_{B^{s -2r+1}_{p,q}} \|u_{k }\|_{B^{s  }_{p,q}} +\|\rho_{k }\|_{B^{s-2r+1 }_{p,q}}^2\right)d\tau.
\end{split}
\end{equation*}
and
\begin{equation*}
\begin{split}
 \|\rho_{k+1}\|_{B^{s-2r +1}_{p,q}} &\leq \exp\left(2C\int_0^t \|  u_{k}  (\tau)\|_{B^{s-2r }_{p,q}}   d\tau\right)\|\rho(0)\|_{B^{s-2r+1}_{p,q}}\\
&+C\int_0^t\exp\left(2C\int_\tau^t \|  u_{k}  (\tau')\|_{B^{s-2r }_{p,q}}  d\tau'\right)\|  u_{k }\|_{B^{s-2r +2}_{p,q}}\|\rho_{k}\|_{B^{s-2r +1}_{p,q}}d\tau.
\end{split}
\end{equation*}
 Thus, adding the two resulted inequalities yields
\begin{equation}
\begin{split}
& \|u_{k+1}\|_{B^{s }_{p,q}}+ \|\rho_{k+1}\|_{B^{s-2r+1 }_{p,q}} \leq \exp\left(2C\int_0^t  \|  u_{k} (\tau)\| _{B^{s }_{p,q}}  d\tau\right)\left(\|u(0)\|_{B^{s }_{p,q}}+\|\rho(0)\|_{B^{s-2r+1 }_{p,q}}\right)\\
&\quad+C\int_0^t\exp\left(2C\int_\tau^t  \|  u_{k} (\tau')\| _{B^{s }_{p,q}}  d\tau'\right)\left(\|\alpha\|  \|u_{k }\| _{B^{s  }_{p,q}}+  \|u_{k }\|^2_{B^{s  }_{p,q}} +\|\rho_{k }\|_{B^{s-2r+1 }_{p,q}}^2+\|  u_{k }\|_{B^{s }_{p,q}}\|\rho_{k}\|_{B^{s-2r +1}_{p,q}}\right)d\tau\\
&\leq \exp\left(2C\int_0^t  \|  u_{k} (\tau)\| _{B^{s }_{p,q}}  d\tau\right)\bigg(\|u(0)\|_{B^{s }_{p,q}}+\|\rho(0)\|_{B^{s-2r+1 }_{p,q}}\bigg)\\
&\quad\quad+C\int_0^t\exp\left(2C\int_\tau^t  \|  u_{k} (\tau')\| _{B^{s }_{p,q}}  d\tau'\right)\left(\|u_{k }\|_{B^{s }_{p,q}} +\|\rho_{k }\|_{B^{s-2r+1 }_{p,q}} +||\alpha|| \right) \left(\|u_{k }\|_{B^{s }_{p,q}} +\|\rho_{k }\|_{B^{s-2r+ 1}_{p,q}}  \right)  d\tau.
\end{split}
\end{equation}
  Choosing
  $$T^*=\left\{
\begin{array}{llll}
 \min\left\{\frac{1}{8 C\left( \|u_0 \|_{B^{s}_{p,q}}+\|\rho_0\|_{B^{s-2r+1}_{p,q}} \right)  },\frac{1}{4C||\alpha||}\right\}, \text { for } ||\alpha||\neq0,\\
 \frac{1}{8  C\left( \|u_0 \|_{B^{s}_{p,q}}+\|\rho_0\|_{B^{s-2r+1}_{p,q}} \right)  },\,\ ||\alpha||=0.
\end{array}
\right.$$
Suppose by induction that for all $t\in [0,T^*)$
\begin{align}
 \|u_{k}(t) \|_{B^{s}_{p,q}}+\|\rho_{k}(t) \|_{B^{s-2r+1}_{p,q}}\leq\frac{ 2(\| u_{0}\|_{B^{s}_{p,q}}+\| \rho_{0}\|_{B^{s-2r+1}_{p,q}}) }{ 1-8  C (\| u_{0}\|_{B^{s}_{p,q}} +\| \rho_{0}\|_{B^{s-2r+1}_{p,q}} )t} = \frac{ 2z_0}{ 1-8 C z_0t} ,
\end{align}
where $z_0=\| u_{0}\|_{B^{s}_{p,q}}+\| \rho_{0}\|_{B^{s-2r+1}_{p,q}} $.

The condition $0\leq\tau<t<T^*\leq \frac{1}{8  C\left( \|u_0 \|_{B^{s}_{p,q}}+\|\rho_0\|_{B^{s-2r+1}_{p,q}} \right)  }$ implies
$$\exp\left(2C\int_\tau^t  \|  u_{k} (\tau')\| _{B^{s }_{p,q}}  d\tau'\right)\leq \exp\left(\int_\tau^t\left( \frac{4Cz_0}{ 1-8 C z_0\tau'}  \right) d\tau'\right)=\left(\frac{1-8 C  z_0 \tau}{1-8 C  z_0 t}\right)^\frac{1}{2 }, $$
and inserting the above inequality and (2.7) into (2.6), we obtain
\begin{align*}
 \|u_{k+1}\|_{B^{s }_{p,q}}+ \|\rho_{k+1}\|_{B^{s-2r+1 }_{p,q}}&\leq\frac{z_0}{ \sqrt{1-8  C  z_0  t } }+\frac{C}{ \sqrt{1-8 C  Z_0  t} } \int^{t}_{0} \left(\frac{4z_0^2}{(1-8  C  z_0 \tau)^\frac{3}{2}}+\frac{2z_0 ||\alpha||}{(1-8 C  z_0 \tau)^\frac{1}{2}}\right) d\tau\\
& =\frac{z_0}{  1-8  C  z_0  t   }+\frac{||\alpha||}{2 \sqrt{1-8 C  z_0  t } }\left(     1-\sqrt{1-8  C  z_0  t }\right)\\
& =\frac{z_0}{  1-8   C  z_0  t   }+\frac{||\alpha||}{2 (1-8   C  z_0  t ) }\left(   \sqrt{1-8   C  z_0  t }  -1+8  C  z_0  t\right)\\
& \leq\frac{z_0}{  1-8   C  z_0  t   }+\frac{4 C ||\alpha|| z_0  t}{  (1-8   C  z_0  t ) } \\
& \leq\frac{ 2z_0}{ 1-8 C z_0t},
\end{align*}
the last inequality we used $t<T^*\leq \frac{1}{4C||\alpha||}$ for $||\alpha||\neq0$.

Now, we prove that $(u_k,\rho_k)_{k\in\mathbb{R}}$ is uniformly bounded in $\mathcal{C}([0,T];B_{p,q}^{s}) \times   \mathcal{C}([0,T];B_{p,q}^{s-2r+1})$ for $T\in(0,T^*)$. Using the equation $(T_k)$ and the similar argument in the proof Lemma 2.1, one can easily prove that $(\partial_tu_k,\partial_t\rho_k)_{k\in\mathbb{R}}$ is uniformly bounded in $\mathcal{C}([0,T];B_{p, r}^{s-1}) \times \mathcal{C}([0,T];B_{p, r}^{s-2r})$, which yields that the sequence $(u_k,\rho_k)_{k\in\mathbb{R}}$ is uniformly bounded in $E_{p, r}^{s }(T) \times E_{p, r}^{s-2r+1 }(T)$.

Let us now show that $(u_k,\rho_k)_{k\in\mathbb{R}}$ is a Cauchy sequence in   $\mathcal{C}([0,T];B_{p,q}^{s-1})\times \mathcal{C}([0,T];B_{p,q}^{s-2r})$. In fact, for all $k,j\in \mathbb{N}$, from $(T_k)$, we have
$$ \left\{
\begin{array}{llll}
(\partial_t+u_{k+j} \partial_x)(m_{k+j+1}-m_{k+1})=F(t,x),\\
(\partial_t+u_{k+j} \partial_x)(\rho_{k+j+1}-\rho_{k+1})=G(t,x),\\
\end{array}
\right.$$
with
\begin{equation*}
\left\{
\begin{array}{llll}
F(t,x)=&- (u_{k+j }-u_{k })   \partial_{x}m_{k+1 } -b (m_{k+j }-m_{k }) \partial_x u_{k+j }    -b m_{k } \partial_x (u_{k+j }-u_{k })\\
  & -\alpha \partial_x(u_{k+j }-u_{k })-\kappa \partial_x[(\rho_{k+j }+\rho_{k })(\rho_{k+j }-\rho_{k })] ,\\
G(t,x)=&- (u_{k+j }-u_{k }) \partial_{x}\rho_{k+1 }  -(b-1)(\rho_{k+j }-\rho_{k })\partial_xu_{k}  -(b-1)\rho_{k } \partial_x(u_{k+j }-u_{k }).
\end{array}
\right.
\end{equation*}

Applying Lemma 2.2 in \cite{ZMW}  again, similar to the proof of Lemma 2.1, yields for every $t\in [0,T)$
\begin{equation}
\begin{split}
 \| u_{k+j+1} -u_{k+1}  \|_{B^{s-1}_{p,q}}\leq &\exp\left(2C\int_0^t  \|  u_{k+j} (\tau)\| _{B^{s }_{p,q}}  d\tau\right) \| u_{k+j+1}(0) -u_{k+1} (0) \|_{B^{s-1}_{p,q}} \\
& +C\int_0^t\exp\left(2C\int_\tau^t  \|  u_{k+j} (\tau')\| _{B^{s }_{p,q}}  d\tau'\right)  \| F(t,\cdot) \|_{B^{s-2r-1}_{p,q}}  d\tau,
\end{split}
\end{equation}
and
\begin{equation}
\begin{split}
  \| \rho_{k+j+1} -\rho_{k+1}  \|_{B^{s-2r}_{p,q}}
&\leq \exp\left(2C\int_0^t  \|  u_{k+j} (\tau)\| _{B^{s  }_{p,q}}  d\tau\right)  \| \rho_{k+j+1}(0) -\rho_{k+1}(0)\|_{B^{s-2r}_{p,q}}  \\
& +C\int_0^t\exp\left(2C\int_\tau^t  \|  u_{k+j} (\tau')\| _{B^{s }_{p,q}}  d\tau'\right)   \| G(t,\cdot) \|_{B^{s-2r }_{p,q}} d\tau.
\end{split}
\end{equation}
For $s>\max\left\{2r +\frac{1}{p}, 2r+1-\frac{1}{p}\right\}$ and $s\neq 2r+2+\frac{1}{p},2r+1+\frac{1}{p}$, we arrive at
\begin{align*}
 \| F(t,\cdot) \|_{B^{s-2r-1}_{p,q}} +& \| G(t,\cdot) \|_{B^{s-2r }_{p,q}}\leq C\left(\| u_{k+j} -u_k  \|_{B^{s-1}_{p,q}}+ \| \rho_{k+j} -\rho_k  \|_{B^{s-2r}_{p,q}} \right)H(t),
\end{align*}
with
\begin{align*}
H(t)=& \| u_{k+1}  \|_{B^{s }_{p,q}}+\| u_{k+j}  \|_{B^{s }_{p,q}}+\| u_{k}  \|_{B^{s }_{p,q}}+\| \alpha \|_{B^{s-2r }_{p,q}}+\| \rho_{k+1}  \|_{B^{s-2r }_{p,q}}+\| \rho_{k+j}  \|_{B^{s-2r }_{p,q}}+\| \rho_{k }  \|_{B^{s-2r }_{p,q}}.
\end{align*}
Submitting above inequality to (2.8)-(2.9), we get
\begin{align*}
V_{k+1}^j(t)&\doteq\| u_{k+j+1} -u_{k+1}  \|_{B^{s-1}_{p,q}}+ \| \rho_{k+j+1} -\rho_{k+1}  \|_{B^{s-2r}_{p,q}} \\
&\leq \exp\left(2C\int_0^t  \|  u_{k+j} (\tau)\| _{B^{s  }_{p,q}}  d\tau\right)\left(\| u_{k+j+1}(0) -u_{k+1} (0) \|_{B^{s-1}_{p,q}}+ \|\rho_{k+j+1}(0) -\rho_{k+1}(0)  \|_{B^{s-2r}_{p,q}} \right.\\
&\quad\left.+C\int_0^t  \exp\left( 2C\int_\tau^t  \|  u_{k+j} (\tau')\| _{B^{s  }_{p,q}}  d\tau'\right)\left(\| F(\tau) \|_{B^{s-2r-1}_{p,q}} + \| G(\tau) \|_{B^{s-2r-1}_{p,q}}\right) d\tau\right)\\
&\leq \exp\left(2C\int_0^t  \|  u_{k+j} (\tau)\| _{B^{s  }_{p,q}}  d\tau\right)\left(V_{k+1}^j(0)   +C\int_0^t  \exp\left(-2C\int_0^\tau  \|  u_{k+j} (\tau')\| _{B^{s  }_{p,q}}  d\tau'\right)V_{k}^j(\tau)
H(\tau) d\tau\right).
\end{align*}
As per Proposition 2.1 in \cite{ZMW}, we have
\begin{equation*}
\begin{split}
\| u_{k+j+1}(0) -u_{k+1} (0) \|_{B^{s-1}_{p,q}}&=\| S_{k+j+1}u(0) -S_{k+1} u(0) \|_{B^{s-1}_{p,q}}=
\left\|\sum_{d=k+1}^{k+j}\Delta_du_0\right\|_{B_{p,q}^{s-1}}\\
&=\left(\sum_{k\geq -1}2^{k(s-1)\sigma}\left\|\Delta_k\left(\sum_{d=k+1}^{k+j}\Delta_du_0\right)\right\|_{L^p}^\sigma\right)^\frac{1}{\sigma}\\
&\leq C\left(\sum_{d=k}^{k+j+1}2^{-d\sigma}2^{d\sigma s}\left\|\Delta_d u_0 \right\|_{L^p}^\sigma\right)^\frac{1}{\sigma}\leq C2^{-k}||u_0||_{B_{p,q}^{s }}.
\end{split}
\end{equation*}
Similarly, we obtain
 $$\| \rho_{k+j+1}(0) -\rho_{k+1} (0) \|_{B^{s-2r}_{p,q}}\leq C2^{-k}||\rho_0||_{B_{p,q}^{s-2r +1}}.$$
Due to $\{u_k,\rho_k\}_{k\in\mathbb{N}}$ being uniformly bounded in $E_{p, r}^{s }(T) \times E_{p, r}^{s-2r+1 }(T)$, one may  find a positive constant $C_T$ independent of $k,j$ such that
$$V_{k+1}^j(t)\leq C_T\left(2^{-k}+\int_0^tV_{k}^j(\tau)d\tau\right), \ \forall t\in [0,T]. $$
%$\text{ for all } t\in [0,T]$.
Employing %Arguing by
the induction procedure with respect to the index $k$ leads to%, we can obtain
\begin{equation*}
\begin{split}
V_{k+1}^j(t)&\leq C_T\left(2^{-k}\sum_{i=0}^k\frac{(2TC_T)^i}{i!}+C_T^{k+1}\int\frac{(t-\tau)^k}{k!}d\tau\right)\leq 2^{-k}\left(C_T\sum_{i=0}^k\frac{(2TC_T)^i}{i!}\right)+C_T  \frac{(TC_T)^{k+1}}{(k+1)!},
\end{split}
\end{equation*}
which reveals the desired result as $k\rightarrow\infty$.
Finally, %we can
applying the interpolation method %, which is similar to the proof in Lemma 2.1,
to the critical case  $s=2r+2+\frac{1}{p}$ concludes the proof of Lemma 2.2.
\end{proof}

Finally, we prove the existence and uniqueness for  Eq.(\ref{Eq.(1.1)}) in Besov space.
\begin{proof} [Proof of Theorem \ref{result1}]
In view of Lemma 2.2, we have that the pair $(u_k,\rho_k)_{k\in\mathbb{N}}$	is a Cauchy sequence in $\mathcal{C}([0,T];B_{p,q}^{s -1 })\times \mathcal{C}([0,T];B_{p,q}^{s -2r})$, thus, it converges to some limit function $(u ,\rho )\in \mathcal{C}([0,T];B_{p,q}^{s -1 })\times \mathcal{C}([0,T];B_{p,q}^{s -2r})$. Next, we show that $(u,\rho)\in E_{p,q}^{s}(T)\times E_{p,q}^{s-2r}(T)$ solves system (1.1).
According to Lemma 2.2, one can see that $(u_k,\rho_k)_{k\in\mathbb{N}}$ is uniformly bounded in $L^\infty([0,T];B_{p,q}^{s })\times L^\infty([0,T];B_{p,q}^{s-2r+1})$. Fatou property for Besov spaces (Proposition 2.2(6) in \cite{ZMW}) guarantees that $(u,\rho)$ also belong to $L^\infty([0,T];B_{p,q}^{s })\times L^\infty([0,T];B_{p,q}^{s-2r+1})$.

On the other hand, as  $(u_k,\rho_k)_{k\in\mathbb{N}}$ converges to $(u,\rho)$ in $\mathcal{C}([0,T];B_{p,q}^{s -1 })\times \mathcal{C}([0,T];B_{p,q}^{s -2r})$,
an interpolation argument insures that the convergence holds in $\mathcal{C}([0,T];B_{p,q}^{s'   })\times \mathcal{C}([0,T];B_{p,q}^{s' -2r+1})$, for any $s'<s$.

Taking limit in $T_k$ reveals
that $(u,\rho)$ satisfy system (1.1) in the sense of $\mathcal{C}([0,T];B_{p,q}^{s'-1 })\times \mathcal{C}([0,T];B_{p,q}^{s'-2r})$ for all $s'<s$.
In view of the fact that $(u,\rho)$ belongs to $L^\infty([0,T];B_{p,q}^{s })\times L^\infty([0,T];B_{p,q}^{s-2r+1})$ and $B_{p,q}^s$ is an algebra as $s>2r+\frac{1}{p}$, we get that the right-hand side of the equation
$$m_{t}+um_x=\alpha   u_x-bu_xm-\kappa \rho \rho_x $$
belongs to $L^\infty([0,T];B_{p,q}^{s })$, and the right-hand side of the second equation
$$\rho_t+u \rho_x=-(b-1)u_x\rho  $$
belongs to $L^\infty([0,T];B_{p,q}^{s-2r+1 })$. In particular, for the case $r<\infty$, Lemma 2.2 yields that $(u ,\rho )\in \mathcal{C}([0,T];B_{p,q}^{s' })\times \mathcal{C}([0,T];B_{p,q}^{s' -2r+1})$ for any $s'<s$. Finally, applying the equation again, we see that  $(\partial_tu ,\partial_t\rho )\in \mathcal{C}([0,T];B_{p,q}^{s-1 })\times \mathcal{C}([0,T];B_{p,q}^{s  -2r })$ if $r<\infty$, and in $L^\infty([0,T];B_{p,q}^{s-1 })\times L^\infty([0,T];B_{p,q}^{s  -2r })$ otherwise. Therefore, the pair  $(u,v)\in E_{p,q}^{s}(T)\times E_{p,q}^{s-2r+1}(T)$.

On the other hand, the continuity with respect to the initial data in
$$\mathcal{C}([0,T];B_{p,q}^{s'}) \cap \mathcal{C}^1([0,T];B_{p,q}^{s'-1})\times \mathcal{C}([0,T];B_{p,q}^{s'-2r+1}) \cap \mathcal{C}^1([0,T];B_{p,q}^{s'-2r})\quad \text {for all } s'<s,$$
can be obtained by Lemma 2.1 and a simple interpolation argument. The case $s'=s$ can be proved
through the use of a sequence of viscosity approximation solutions
$(u_\epsilon,v_\epsilon)_{\epsilon>0}$	   for System (1.1) which
converges uniformly in $$\mathcal{C}([0,T];B_{p,q}^{s }) \cap \mathcal{C}^1([0,T];B_{p,q}^{s -1})\times \mathcal{C}([0,T];B_{p,q}^{s -2r+1}) \cap \mathcal{C}^1([0,T];B_{p,q}^{s -2r})$$
gives the continuity of solution $(u,\rho)$ in $(u,v)\in E_{p,q}^{s}(T)\times E_{p,q}^{s-2r+1}(T)$.
\end{proof}

\subsection{Critical case}
It is well known that $B_{2,2}^s(\mathbb{R})=H^s$ and for any $s' <2r+\frac{1}{2}<s$:$$H^s\hookrightarrow B^{2r+\frac{1}{2}}_{2,1}\hookrightarrow
H^{2r+\frac{1}{2}}\hookrightarrow
B^{2r+\frac{1}{2}}_{2,\infty}\hookrightarrow H^{s'}, $$ which shows that $H^s$ and $B_{2,1}^s$ are quite close. So, attention is now restricted to the critical case in the local well-posedness.

\begin{thm}\label{result2}
Suppose that    the function $\alpha(x)\equiv c\in \mathbb{R}$ or $\alpha(\cdot)\in B_{2,1}^{1/2}$, and the initial data $z_0\doteq(u_0,\rho_0)\in B_{2,1}^{2r+\frac{1}{2}}\times B_{2,1}^{\frac{3}{2}}$.  Then there exists a time $T>0$ such that the Cauchy problem (1.1) has a unique solution $$z= z(\cdot,z_0)\in \mathcal{C}([0,T];B_{2,1}^{2r+\frac{1}{2}}) \cap \mathcal{C}^1([0,T];B_{2,1}^{2r })\times \mathcal{C}([0,T];B_{2,1}^{\frac{3}{2}}) \cap \mathcal{C}^1([0,T];B_{2,1}^{\frac{1}{2}}).$$
 Moreover, the solution depends continuously on the initial data, i.e. the mapping
$$z_0\mapsto z(\cdot,z_0):B_{2,1}^{2r+\frac{1}{2}} \times B_{2,1}^{\frac{3}{2}}\mapsto\mathcal{C}([0,T];B_{2,1}^{2r+\frac{1}{2}}) \cap \mathcal{C}^1([0,T];B_{2,1}^{2r })\times \mathcal{C}([0,T];B_{2,1}^{\frac{3}{2}}) \cap \mathcal{C}^1([0,T];B_{2,1}^{\frac{1}{2}})$$
 is continuous.
\end{thm}

\begin{proof}
On account of $(u_0,\rho_0)\in B_{2,1}^{2r+\frac{1}{2}}\times B_{2,1}^{\frac{3}{2}}$, the transport theory (see Lemma 2.2 in \cite{ZMW}) can be applied. Similar to the
case  $(u_0,\rho_0)\in B_{p,q}^{s}\times B_{p,q}^{s-2r+1}$  with $s>\max\left\{2r +\frac{1}{p},2r +\frac{1}{2},2r+1-\frac{1}{p}\right\}$ and the proof of Theorem 1 in \cite{D23}, we can establish this result. The proof of the theorem is therefore
omitted without details.
\end{proof}

\section{Compact support}

Next, let us consider the following initial value problem:
\begin{equation}
\left\{
\begin{array}{llll}
\varphi_t= u (t,\varphi(t,x)), &t\in[0,T),x\in\mathbb{R},\\
\varphi(0,x)=x, &x\in\mathbb{R},
\end{array}
\right.
\end{equation}
where $u$ is the first component of the solution $(u,\rho)$ for the problem (1.1).
\begin{lm}(see \cite{ZhouMuZeng})
Let $(u_0,\rho_0)\in H^s\times H^{s-2r+1}$ with $s> 2r+1/2$, and $T>0$ be the life span of the solution to Eq.(1.1).
Then there exists a unique solution $ \varphi\in \mathcal{C}^1([0,T),\mathbb{R})$ to Eq. (3.1).
Moreover, the map $ \varphi(t,\cdot)$ is an increasing diffeomorphism  over $\mathbb{R}$, where
\begin{equation}
\varphi_x (t,x)=\exp\left\{\int_0^t  u _\varphi(s,\varphi(s,x))ds\right\}>0,
 \end{equation}
for all $(t,x)\in[0,T)\times\mathbb{R}$.
\end{lm}

Our next result shows the component $\rho$ of the solution for system (1.1) retains the
property of being compactly supported as it evolves over time, totally independent
of the form of the initial data $u_0$ (and $m_0$). In the following we assume that $(u_0, \rho_0)\in H^s\times H^{s-2r+1}$ with $s> 2r+1/2$, and so $\rho$ is a classical solution of the system.
\begin{lm}
Let $(u_0,\rho_0)\in H^s\times H^{s-2r+1}$ with $s> 2r+1/2$, and  $T>0$ be the maximal existence time of the corresponding solution $ (u, \rho)$ to system (1.1).  Then, we have
\begin{equation}
 \rho(t,\varphi(t,x))\varphi_x^{b-1}=\rho(0,x).
\end{equation}

Moreover, if there exist $M_1>0$    such that
   $ (b-1)u_\varphi(t,\varphi)\geq -M_1$  for all $ (t,x)\in[0,T)\times\mathbb{R}$, then
\begin{equation}
||\rho(t,\cdot)||_{L^\infty}=||\rho(t,\varphi(t,\cdot))||_{L^\infty}\leq \exp\{ M_1T\}||\rho_0(\cdot)||_{L^\infty}, \quad \text{ for all } t\in[0,T)
\end{equation}

Furthermore,   if  $\int_\mathbb{R}|\rho_0(x)|^\frac{1}{b-1}dx$ converge with $b\neq1$, then
$$\int_\mathbb{R}|\rho(t,x)|^\frac{1}{b-1}dx=\int_\mathbb{R}|\rho_0(x)|^\frac{1}{b-1}dx \text{ for all } t\in[0,T),$$

\end{lm}
\begin{proof}
Noticing $\frac{d\varphi_x (t,x)}{dt}=\varphi_{xt} =   u _\varphi (t,\varphi(t,x))\varphi_x (t,x)$, differentiating the left-hand side of the  equation in (3.3) with respect to $t$, and using the second equation in (1.1), we obtain
\begin{equation*}
\begin{split}
\frac{d}{dt}\left(\rho(t,\varphi(t,x))\varphi_x^{b-1}\right)&= \rho_t(t,\varphi(t,x))\varphi_x^{b-1}+\rho_\varphi(t,\varphi(t,x))\varphi_t\varphi_x^{b-1}+(b-1)\rho(t,\varphi(t,x))\varphi_x^{b-2}\varphi_{xt}\\
&=\varphi_x^{b-1}\bigg(\rho_t(t,\varphi(t,x)) +\rho_\varphi(t,\varphi(t,x))u (t,\varphi(t,x))+(b-1)\rho(t,\varphi(t,x))u _\varphi (t,\varphi(t,x))\bigg)\\
&=0.
\end{split}
\end{equation*}
which means that $\rho(t,\varphi(t,x))\varphi_x^{b-1}$ is independent on the time $t$.
By (3.1), we know $\varphi  (0,x)=x$ and $\varphi_x (x,0)=1$, thus, Eq.(3.3) holds.

By Lemma 3.1,  in view of
Eq. (3.3), and $\varphi_x (0,x)=1$
we have
\begin{equation*}
\begin{split}
||\rho(t,\cdot)||_{L^\infty}&=||\rho(t,\varphi (t,\cdot))||_{L^\infty}=|| \varphi_x ^{1-b}\rho_0||_{L^\infty}\\
&=\left\|\exp\left\{ (1-b)\int_0^t u_\varphi(s,\varphi(s,x))ds\right\}\rho_0(\cdot)\right\|_{L^\infty}\\
&\leq \exp\{ M_1T\}||\rho_0(\cdot)||_{L^\infty} \quad\text{ for all } t\in[0,T),
\end{split}
\end{equation*}
and
 \begin{equation*}
\begin{split}
\int_\mathbb{R}|\rho_0(x)|^{\frac{1}{b-1}}dx&=\int_\mathbb{R}|\rho(t,\varphi (t,x))|^{\frac{1}{b-1}}\varphi_x (t,x)dx=\int_\mathbb{R}|\rho(t,\varphi (t,x))|^{\frac{1}{b-1}}d\varphi(t,x)\\
&=\int_\mathbb{R}|\rho(t,x)|^{\frac{1}{b-1}}dx \text{ for all } t\in[0,T),
\end{split}
\end{equation*}
which guarantee the lemma is true.
\end{proof}
\begin{cor}
If $\rho(t,x)$ is a solution for (1.1) with initial data  $\rho_0=\rho(0,x)$ then we have
the following representation,
\begin{equation}
 \rho(t,x) =\rho_0(\varphi^{-1}(t,x))e^{ (1-b)\int_0^tu_x(s,x)ds}, \,\ t\in[0,T).
\end{equation}
Moreover, if $\rho_0=\rho(0,x)$ is supported in the interval $[\beta_{\rho_0},\gamma_{\rho_0}]$, then $\rho$ will remain
compactly supported for all further times of its existence $t\in[0,T)$, with the support
of $\rho(t,x)$ contained in the interval $[\varphi(t,\beta_{\rho_0}),\varphi(t,\gamma_{\rho_0})]$.

\end{cor}
\begin{proof}
The relation (3.5) follows from the identities (3.3) and (3.2) and the invertibility of
$\varphi$ for each $t\in[0,T)$. Furthermore, since $\varphi_x > 0$, and $\rho_0$ is compactly supported, it follows from this relation
that $\rho(t,x)$ is compactly supported for all times $t\in[0,T)$, with support contained
in $[\varphi(t,\beta_{\rho_0}),\varphi(t,\gamma_{\rho_0})]$.
\end{proof}
The next result proves a similar property for solutions $m$ of (1.1), namely that
a solution which is initially compactly supported retains this property throughout
its evolution. However, whereas Corollary 3.1 made no assumptions on the form
of the initial data $u_0,m_0$ aside from the required level of smoothness, Corollary 3.2
below requires that $\rho_0$   be also compactly supported.
\begin{cor}
  Assume that $u_0$ is such that $m=Au_0$ has compact support,
contained in the interval $[\beta_{m_0},\gamma_{m_0}]$, and that $\rho_0$ is also compactly supported,
with support contained in $[\beta_{\rho_0},\gamma_{\rho_0}]$. If $T=T(u_0,\rho_0)>0$ is the maximal existence
time of the unique classical solutions $u(t,x), \rho(t,x)$ to the system of equations (1.1) with the given initial data $(u_0(x) ,\rho_0(x))$ and $\alpha\equiv0$, then for any $t\in[0,T)$ the $C^1$
function $x\mapsto m(t,x)$ has compact support.
\end{cor}
\begin{proof}
Applying the family of diffeomorphisms (3.1) and the first equation in (1.1), we get
\begin{equation*}
\begin{split}
\frac{d}{dt}& \left(m(t,\varphi(t,x))\varphi_x^b(t,x) \right)\\
&=m_t(t,\varphi(t,x))\varphi_x^{b }+m_\varphi(t,\varphi(t,x))\varphi_t\varphi_x^{b }
+ b  m(t,\varphi(t,x))\varphi_x^{b-1}\varphi_{xt}  \\
&=\varphi_x^{b }\bigg(m_t(t,\varphi(t,x)) +m_\varphi(t,\varphi(t,x))\varphi_t+bm(t,\varphi(t,x))u _\varphi (t,\varphi(t,x)) \bigg)\\
&=-\kappa\varphi_x^{b }(t,x) \rho(t,\varphi(t,x))\rho_\varphi(t,\varphi(t,x)).
\end{split}
\end{equation*}
Therefore
\begin{equation}
\begin{split}
m(t,\varphi(t,x))\varphi_x^b(t,x) =m_0(x)  -\kappa\int_0^t\rho(s,\varphi(t,x))\rho_\varphi(s,\varphi(t,x))\varphi_x^{b }(s,x) ds
\end{split}
\end{equation}
Now, by assumption, $m_0(x)$ is supported in the compact interval   $[\beta_{m_0},\gamma_{m_0}]$. For
fixed $t\in[0,T)$, Corollary 3.1 assures us that $\rho(t,\cdot)$ is compactly supported in the
interval $[\beta_{\rho_0},\gamma_{\rho_0}]$, and therefore the integral term in (3.6), regarded as a
function of $x$, must also be compactly supported in the interval  $[\beta_{\rho_0},\gamma_{\rho_0}]$. Therefore,
setting $\alpha=\max\{\alpha_{m_0},\alpha_{\rho_0}\},\gamma=\min\{\gamma_{m_0},\gamma_{\rho_0}\},$
it follows that $m(t,\cdot)$ is compactly supported, with its support contained in the
interval  $[\varphi(t,\beta),\varphi(t,\gamma)]$, for all $t\in[0,T)$.
\end{proof}

  \section{On the Persistence Properties of the  Eq.(1.3)  }

 In   Section 3, we prove  that solutions of (1.3) have infinite propagation speed in the sense that (non-zero)
compactly supported initial data lead to solutions that are not compactly supported for any positive time. In \cite{GHY}, the authors shown that the
solutions possess exponentially decaying profiles for large values of the spatial variable.
Motivated by the recent works \cite{Br,Zhou} on the nonlinear   Camassa-Holm equation in weighted Sobolev spaces. The other aim of this paper is to establish the  persistence properties for the Eq.(1.3)   in weighted $L^p$ spaces.  Our results  generalize the
work presented in  \cite{GHY} on infinite propagation speed and asymptotic profiles to Eq.(1.3) in $L^p_{ \phi}$.

  In  the present section, we intend to find a large class of weight functions $\phi$ such that
\begin{equation*}
\sup_{t\in[0,T)}\left(||u(t)\phi||_{L^p} +||  u_x(t) \phi||_{L^p}+||\rho(t)\phi||_{L^p}\right)<\infty,
\end{equation*}
 this way we obtain a persistence result on solutions $(u,\rho)$ to Eq. (1.3) in the weight $L^p$ spaces   $L^p_{ \phi}:=L^p(\mathbb{R},\phi^p(x)dx)$. As a consequence and an application we determine the spatial asymptotic behavior of certain solutions
to Eq. (1.3).   We will work with moderate weight functions which
appear with regularity in the theory of time-frequency analysis  and have led
to optimal results for the  Camassa-Holm equation in \cite{Br}, we first give the definition for admissible   weight  function. The pre-defined terminologies like $v$-moderate,
sub-multiplicative  can be refer the reader to \cite{Br,Zhou}).

\begin{df}
An admissible weight function for the Eq. (1.3) is a
locally absolutely continuous function $\phi:\mathbb{R}\rightarrow\mathbb{R}$ such that, for some $A>0$ and  almost all  $x\in\mathbb{R}$,
$|\phi'(x)|\leq A|\phi(x)|$, and that is $v$-moderate for some sub-multiplicative weight function $v$
satisfying $\inf_\mathbb{R}v>0$ and
\begin{equation}
\int_\mathbb{R}\frac{v(x)}{e^{|x|}}dx <\infty.
\end{equation}
\end{df}
We can now state our main result on admissible weights.
\begin{thm}\label{result5}
Assume that
$u_0\phi,  u_{0,x}\phi,  \rho_0\phi\in   L^p(\mathbb{R}) $, $1\leq p\leq\infty$ for
 an admissible weight function  $\phi$ of Eq. (1.3),  and $T$ be the maximal time of the solution $z=(u,\rho) $ to system (1.3) with the initial data $z_0=(u_0,\rho_0)\in H^s(\mathbb{R})\times H^{s-1}(\mathbb{R}) $, $s>3/2$.
Then, for all
$t\in [0, T]$, there is a constant $C >0$ depending only on the weight $ \phi$ such that
$$||u(t)\phi||_{L^p} +||u_x(t) \phi||_{L^p}+||\rho(t)\phi||_{L^p}\leq(||u_0\phi||_{L^p} +||u_{0,x}\phi||_{L^p}+||\rho_0\phi||_{L^p})\exp\left\{C(1+M)t\right\}$$
where
$$M\doteq\sup_{t\in[0,T]}\left(||u(t)||_ {L^\infty}+||u_x(t)|| _{L^\infty}+||\rho(t)||_{L^\infty}\right)<\infty.$$
\end{thm}

If we take the standard weights $\phi=\phi_{a,b,c,d}(x)=e^{a|x|^b}(1+|x|)^c\log (e+|x|)^d$   with the following conditions:
$$a\geq0,\,\ c,d\in\mathbb{R},\,\ 0\leq b\leq1,\,\ ab<1,$$
then, the restriction $ab<1$ guarantees the validity of condition (4.1) for a multiplicative function $v(x)\geq1$.
Thus, the classical example of the application for  Theorem 4.1  is the following special persistence properties.

\begin{rem}

(1) Take $\phi=\phi_{0,0,c,0}$  with $c>0$, and choose $p=\infty$. In this case the Theorem 4.1
states that the condition
$$|u_0(x)|+|u_{0,x}(x)|+|\rho_0(x)|\leq C(1+|x|)^{-c}$$
implies the uniform algebraic decay in $[0, T)$:
$$|u (x,t)|+|u_x(x,t)|+|\rho(x,t)|\leq C'(1+|x|)^{-c}.$$
 Thus, we obtain the algebraic decay rates of
strong solutions to the Eq. (1.3).   A corresponding result on algebraic decay rates of strong solutions for the CH equation   can be found in \cite{NZ}.

(2) Choose $\phi=\phi_{a,1,0,0}$  if $x\geq0$, and $\phi(x)=1$ if $x\leq0$ with $0\leq a<1$. It is easy to see that such weight
  satisfies the admissibility conditions of Definition 4.1. Let further $p=\infty$ in
 Theorem 4.1, Then one deduces that Eq. (1.3) preserve  the pointwise decay $O(e^{-ax})$ as
$x\rightarrow+\infty$ for any $t>0$. Similarly, we have persistence
of the decay $O(e^{-ax})$ as $x\rightarrow-\infty$. Theorem 4.1 thus generalizes the main result of \cite{GHY,HMPZ} on persistence of strong solutions to the CH type equations. It is worthwhile to note that Yin et al. also use weight functions and Gronwall's Lemma to obtain the exponential decay [\cite{GHY}, Theorem 5.1], in which they require the Sobolev index $s>5/2$, here we extend to $s>3/2$.

(3)  Our result  generalize the work of \cite{Br} on persistence and non-persistence of solutions to its supersymmetric extension. Meanwhile, since the  dispersive term present in Eq.(1.3), we should improve the corresponding result [\cite{Br}, Theorem 2.2] to $p\in[1,\infty]$.
\end{rem}
Clearly, the limit case $\phi=\phi_{1,1,c,d}$ is not covered by Theorem 4.1.  In the following theorem however we may choose the weight $\phi=\phi_{1,1,c,d}$ with $c<0,d\in\mathbb{R}$,  and $\frac{1}{|c|}<p\leq\infty$, or more generally
when $ (1+|\cdot|)^c\log (e+|\cdot|)^d\in L^p(\mathbb{R})$. See   Corollary 4.1 below:

\begin{cor}\label{result6}
 Let   $1\leq p \leq \infty$  and $\phi$ be a $v$-moderate weight function as in Definition 4.1
satisfying
\begin{equation}
v e^{-|\cdot|}\in L^p (\mathbb{R}).
\end{equation}
Assume that
$ u_0\phi,u_{0,x}\phi,\rho_0\phi\in L^p(\mathbb{R})$ and
$ u_0\phi^{1/2},u_{0,x}\phi^{1/2},\rho_0\phi^{1/2}\in L^2(\mathbb{R})$.
Let also $z=(u,\rho)\in \mathcal{C}([0,T),H^s(\mathbb{R}))\times \mathcal{C}([0,T),H^{s-1}(\mathbb{R})), s>3/2$  be the strong solution of the Cauchy problem for
Equation (1.1)  emanating from $z_0=(u_0,\rho_0)$. Then,
$$\sup_{t\in[0,T)}\left(||u(t)\phi||_{L^p}+||u_x(t)\phi||_{L^p}+||\rho(t) \phi||_{L^p}\right)$$
and
$$\sup_{t\in[0,T)}\left(||u (t)\phi^\frac{1}{2}||_{L^{2}}+||u_x(t) \phi^\frac{1}{2}||_{L^{2}}+||\rho(t) \phi^\frac{1}{2}||_{L^{2}}\right)$$
are finite.
\end{cor}

 \begin{rem}
For the particular choice $c=d=0$ and $p=\infty$, we conclude from $|u_0(x)|+|u_{0,x}(x)|+|\rho_0(x)|\leq Ce^{-|x|}$ for any $x\in \mathbb{R}$ that the unique solution $z=(u,\rho)\in \mathcal{C}([0,T);H^s(\mathbb{R}))\times \mathcal{C}([0,T);H^{s-1}(\mathbb{R}))$ of (1.3) with $z_0=(u_0,\rho_0)$ satisfies
\begin{align}
|u(x,t)|+|\partial_x u(x,t)|+|\rho(x,t)|\leq C'e^{-|x|}
\end{align}
uniformly in $\mathbb{R}\times[0, T)$.
\end{rem}
First, we present some standard definitions.  In general a weight function is simply a non-negative function. A weight function $v:\mathbb{R}^n\rightarrow\mathbb{R}$ is called sub-multiplicative if
$$v(x+y)\leq v(x)v(y), \text{ for all } x,y \in \mathbb{R}^n.$$
Given a sub-multiplicative function $v$, a positive function $\phi$ is $v$-moderate if and only if
$$\exists C_0>0:\phi(x+y)\leq C_0v(x)\phi(y),\text{ for all } x,y \in \mathbb{R}^n.$$
If $\phi$ is $v$-moderate for some sub-multiplicative function $v$, then we say that $\phi$ is moderate. This is the usual terminology in time-frequency analysis papers \cite{Br}. Let us recall the most standard examples of such weights. Let
\begin{equation*}
\phi(x)=\phi_{a,b,c,d}(x)=e^{a|x|^b}(1+|x|)^c\log (e+|x|)^d.
\end{equation*}
We have (see \cite{Br}) the following conditions:

(i) For $a, c,d\geq0$ and $0\leq b\leq1$ such weight is sub-multiplicative.

(ii) If $a, c,d\in\mathbb{R}$ and $0\leq b\leq1$, then $\phi$ is moderate. More precisely, $\phi_{a,b,c,d}$ is $\phi_{\alpha,\beta,\gamma,\delta}$-moderate for $|a|\leq\alpha$,$|b|\leq\beta$, $|c|\leq\gamma$ and $|d|\leq\delta$.

The elementary properties of sub-multiplicative and moderate weights can be find in \cite{Br}. Now,   we prove Theorem 4.1.

\begin{proof}[Proof of Theorem \ref{result5}]
Let $z=(u,\rho)$ be the solution to Eq.(1.3) with initial data $z_0\in H^s\times H^{s-1}$, $s>3/2$, and $T$ be the maximal existence times of the solution $z$, which is guaranteed by Theorem 3.2 in \cite{GHY}.

From the solution  $z=(u,\rho)\in \mathcal{C}([0,T),H^s(\mathbb{R}))\times \mathcal{C}([0,T),H^{s-1}(\mathbb{R})),s>3/2$, The Sobolev's imbedding theorem yields
$$M\equiv \sup_{t\in[0,T]}\left(||u(t,\cdot)||_{L^\infty}+||  u_x(t,\cdot)||_{L^\infty} +||\rho(t,\cdot)||_{L^\infty}\right)<\infty.$$
For any $N\in\mathbb{Z}^+$, let us consider the $N$-truncations of $\phi(x)$:
$f(x)=f_N(x)=\min\{\phi(x),N\}$.
Then $f : \mathbb{R}\rightarrow\mathbb{R}$ is a locally absolutely continuous function such that
$$||f||_\infty\leq N,\quad |f'(x)|\leq A|f(x)|\quad \text{ a.e. on } \mathbb{R}.$$
In addition, if $C_1=\max\{C_0,\alpha^{-1}\}$, where $\alpha=\inf_{x\in\mathbb{R}}v(x)>0$, then
$$f(x+y)\leq C_1 v(x)f(y), \quad \forall x,y\in \mathbb{R}.$$
Moreover, as shown in \cite{Br},    the $N$-truncations $f$ of a
$v$-moderate weight $\phi$ are uniformly $v$-moderate with respect to $N$.

Recall that the operator $ \Lambda^{-2}$, acting on $L^2(\mathbb{R})$, can be expressed in its associated
Green's function $G(x)=\frac{1}{2}e^{-|x|}$. So, we can rewrite  Eq.(2.2) as the following nonlocal form
\begin{equation}
\left\{
\begin{array}{llll}
u_t+ uu_x=- \partial_xG*P(u,\rho),&t>0,x\in\mathbb{R},\\
\rho_t+u \rho_x=-(b-1)u_x\rho ,&t>0,x\in\mathbb{R},\\
u(x,0)=u_0(x),\rho(x,0)=\rho_0(x),&x\in\mathbb{R}.
\end{array}
\right.
\end{equation}
with $P(u,\rho)\doteq\frac{b}{2}u^2+\frac{3-b}{2}u_x^2+\frac{\kappa}{2}\rho^2-\alpha u$.

Let us start from the case $1\leq p <\infty$.
Multiplying the first equation in Eq. (4.4) by $|uf|^{p-1} sgn (uf)f$ and integrating it lead to % to obtain
\begin{equation}
\begin{split}
\int_\mathbb{R}|uf|^{p-1} sgn (uf)(\partial_tu f)dx=&\int_\mathbb{R}|uf|^{p-1} sgn (uf)fuu_x dx - \int_\mathbb{R} |uf|^{p-1} sgn (uf)f\cdot \partial_x(G*P(u,\rho))dx,
\end{split}
\end{equation}
The first term on the left hand of (4.5) reads
$$\int_\mathbb{R}|uf|^{p-1} sgn (uf)(\partial_tu f)dx=\frac{1}{p}\frac{d}{dt}||uf||_{L^{p}}^{p}=||uf||_{L^{p}}^{p-1}\frac{d}{dt}||uf||_{L^{p}}.$$
Then, the H\"{o}lder's inequality is followed by the estimate
\begin{align*}
\left|\int_\mathbb{R}|uf|^{p-1} sgn (uf)uu_xfdx\right|&\leq ||uf||_{L^p}^{p-1}|| u  u_xf||_{L^{p}}\leq   M ||uf||_{L^p}^{p-1}|| u_x  f||_{L^{p}}.
\end{align*}
For the nonlocal term,  we have%get
\begin{align*}
 &\left|\int_\mathbb{R} |uf|^{p-1} sgn (uf)[f\cdot\partial_x(G*P(u,\rho)) ]dx\right|\leq ||uf||_{L^{p}}^{p-1} \left\|f\cdot\partial_x(G*P(u,\rho)) \right\|_{L^{p}}\\
&\quad\leq C_{\alpha,b,\kappa}||uf||_{L^{p}}^{p-1}\left\{ ||(\partial_xG)v||_{L^{1}}\left\|f\cdot\left( u+ u^2 +  u_x^2+\rho^2\right)\right\|_{L^{p}}\right\}\\
&\quad\leq  C (1 +M ) ||uf||_{L^{p}}^{p-1} ( ||uf||_{L^p}+  || u_x f||_{L^p}+  || \rho f||_{L^p}),
\end{align*}
where the H\"{o}lder's inequality, Proposition 3.1 and 3.2 in \cite{Br}, and Condition (4.1) are applied in the first inequality, the second one, and the last one, respectively,
  %we applied Proposition 3.1 and 3.2 in \cite{Br}, and the last we used condition (1.11) In the first inequality we used H\"{o}lder's inequality, and in the second inequality we applied Proposition 3.1 and 3.2 in \cite{Br}, and the last we used condition (1.11).
and the constant $C$ only depends on $v$ and $\phi$.
Form Eq. (4.5) one may get% we can obtain
\begin{equation}
\frac{d}{dt}||uf||_{L^{p}}\leq C    (1+ M)  (||uf||_{L^{p}} +  || u_x f||_{L^p}+  || \rho f||_{L^p}).
\end{equation}
Let us now give the estimate on $ u_xf $. Differentiating the first equation in Eq. (4.4) with respect to the variable $x$ and then multiplying by $f$, we may arrive at% produces the equation
\begin{equation*}
\begin{split}
\partial_t(u_x f)+u  u_{xx} f+ u_x^2f + f \partial_x^2(G* P(u, \rho) ) =0,
\end{split}
\end{equation*}
which yields %Multiplying this equation by $|u_xf|^{p-2}(u_xf)$ with $p\in \mathbb{Z}^+$, integrating the result in the $x$-variable. It is easy to see that
$$\int_\mathbb{R}| u_xf|^{p-1}sgn(u_xf )\partial_t( u_x f)dx= ||u_xf||_{L^{p}}^{p-1}\frac{d}{dt}||u_xf||_{L^{p}},$$
$$\left| \int_\mathbb{R}|u_xf|^{p-1}sgn(u_xf)fu_x^2dx \right|\leq ||u_xf||_{L^p}^{p-1}|| u_x  f  u_x||_{L^{p}}\leq M||u_xf||_{L^p}^{p-1}|| u_x f||_{L^{p}},$$
and
\begin{align*}
&\left|\int_\mathbb{R} |u_xf|^{p-1}sgn(u_xf)[f  \partial_x^2(G* P(u,\rho) ) ]dx\right|\leq ||u_xf||_{L^{p}}^{p-1} ||f \partial_x^2(G* P(u, \rho) ) ||_{L^{p}}\\
&\quad\quad\leq C   (1+ M) ||u_xf||_{L^{p}}^{p-1}\left( ||uf||_{L^p}+ ||u_xf||_{L^p}+ ||\rho f||_{L^p}\right).
\end{align*}
For the second order derivative term, we have%get
\begin{equation*}
\begin{split}
&\left|\int_\mathbb{R}|u_xf|^{p-1}sgn(u_xf) u  u_{xx}fdx\right|\\
&\quad=\left|\int_\mathbb{R}|u_xf|^{p-1}sgn(u_xf) u\left[\partial_x( u_xf)- u_x  f_x )\right]dx\right|\\
&\quad=\left|\int_\mathbb{R} u \partial_x\left(\frac{|u_xf|^p}{p}\right)dx-\int_\mathbb{R}|u_xf|^{p-1}sgn(u_xf) u u_x f_x  dx\right|\\
&\quad\leq  M(1+A)||u_xf  ||_{L^p}^p,
\end{split}
\end{equation*}
where the inequality $| f_x(x)|\leq A f(x)$, for a.e. $x$, is applied.
Thus, it follows that%we get
\begin{equation}
\frac{d}{dt}||u_xf||_{L^{p}}\leq C_3(1+M) (||uf||_{L^p}+ || u_x f||_{L^p}+ ||\rho f||_{L^p}).
\end{equation}
We now multiply the second equation in Eq.(4.4) with $|\rho f|^{p-1} sgn (\rho f)f$ and integrate to obtain the identity
\begin{equation*}
\frac{1}{p}\frac{d}{dt}||\rho f||_{L^{p}}^p+ \int_\mathbb{R}|\rho f|^{p-1} sgn (\rho f)f u \rho_x dx+(b-1)\int_\mathbb{R}|\rho f|^{p-1} sgn (\rho f)f u_x \rho dx=0.
\end{equation*}
As above, we get
\begin{equation*}
\begin{split}
&\left|\int_\mathbb{R}|\rho f|^{p-1} sgn (\rho f)f u_x \rho dx\right|\leq ||\rho f||_{L^p}^{p-1}|| fu_x   \rho||_{L^{p}}\leq M||\rho f||_{L^p}^{p-1}|| \rho f||_{L^{p}},
\end{split}
\end{equation*}
and
\begin{equation*}
\begin{split}
 \left|\int_\mathbb{R}|\rho f|^{p-1} sgn (\rho f)f u \rho_x dx \right|& =\left|\int_\mathbb{R}|\rho f|^{p-1} sgn (\rho f) u\left[\partial_x( \rho f)- \rho  f_x )\right]dx\right|\\
& =\left|\int_\mathbb{R} u \partial_x\left(\frac{|\rho f|^p}{p}\right)dx-\int_\mathbb{R}|\rho f|^{p-1}sgn(\rho f) u \rho f_x  dx\right|\\
& \leq  M(1+A)||\rho f  ||_{L^p}^p,
\end{split}
\end{equation*}
this yields
\begin{equation}
\frac{d}{dt}||\rho f||_{L^{p}}\leq C_4 M  ||\rho f||_{L^p} .
\end{equation}
Based on the inequalities (4.6)-(4.8), by Gronwall's inequality
$$|| u(t)f||_{L^{p}}+||u_x (t)f||_{L^{p}}+||\rho(t)f||_{L^{p}}\leq \left(|| u_0f||_{L^{p}}+|| u_{0,x}f||_{L^{p}}+||\rho_0f||_{L^{p}}\right)\exp\left(C(1+M) t\right), \text{ for all } t\in [0,T).$$
Since $f(x)=f_N(x)\uparrow \phi(x)$ as $N\rightarrow \infty$ for a.e. $x\in \mathbb{R}$ and $u_0\phi,  u_{0,x}\phi,\rho_0\phi\in L^p(\mathbb{R})$ the assertion
of the theorem follows for the case $p\in[1,\infty)$. Since $||\cdot||_{L^\infty}=\lim_{p\rightarrow\infty}||\cdot||_{L^p}$ it is clear that the theorem also applies for $p=\infty$.
\end{proof}

\begin{proof}[Proof of Corollary \ref{result6}]
As explained in \cite{Br}, if the function  $\phi$ is a $v$-moderate weight function, then the function $\phi^{1/2}$ is also a $v^{1/2}$-moderate weight satisfying $|(\phi^{1/2})'(x)|\leq \frac{A}{2}\phi^{1/2}(x)$, $\inf v^{1/2}>0$ and $v^{1/2}e^{-|\cdot|}\in L^1(\mathbb{R})$. We use Theorem 4.1 with $p=2$ to the weight $\phi^{1/2}$ and obtain
\begin{equation}
|| u(t)\phi^\frac{1}{2}||_{L^{2}}+||u_x (t)\phi^\frac{1}{2}||_{L^{2}}+||\rho(t)\phi^\frac{1}{2}||_{L^{2}}\leq \left(|| u_0\phi^\frac{1}{2}||_{L^{2}}+|| u_{0,x}\phi^\frac{1}{2}||_{L^{2}}+||\rho_0\phi^\frac{1}{2}||_{L^{2}}\right)\exp\left(C(1+M) t\right).
\end{equation}
In view of Proposition 3.2 in \cite{Br}, noticing $f(x)=f_N(x)=\min \{\phi(x),N\}$ admits%, we obtain
\begin{equation}
\begin{split}
||f\partial_x&(G*P(u, u_x))||_{L^p}||_{L^p}   \leq C_{\alpha,\kappa,b}||f\partial_x \left(G(x)* ( u+ u^2 +u_x ^2+\rho^2)\right)||_{L^p}\\
&\leq C||f \partial_x G(x) ||_{L^p}||f ( u+ u^2+u^3+ u^4+ u_x ^2)||_{L^1}\\
&\leq C||fe^{-|x|}   ||_{L^p}\left(||f   u||_{L^1}+||f ^\frac{1}{2}u ||_{L^2}^2 +||f ^\frac{1}{2} u_x ||_{L^2}^2+||f ^\frac{1}{2}\rho ||_{L^2}^2\right)\\
&\leq C_1\exp\left(C_2(1+M) t\right),
\end{split}
\end{equation}
where we used (4.9) and Theorem 4.1 with $p=1$.

Similarly, noticing $\partial_x^2 G=G-\delta$ reveals%,
\begin{equation}
\begin{split}
||f\partial_x^2(G*P(u, u_x))||_{L^p}
&\leq C_1\exp\left(C_2(1+M) t\right)+C_3(1+M) (||uf||_{L^p}+||fu_x||_{L^p}+||f\rho||_{L^p}),
\end{split}
\end{equation}
where the constants on the right-hand side of Eqs. (4.10) and (4.11) are independent of $N$.

By using the procedure as shown in the proof of Theorem 4.1, we can readily obtain%easy get,
\begin{align}
\frac{d}{dt}||uf||_{L^p}\leq C(1+M) ||uf||_{L^p}+||f \partial_x(G*G(u, u_x))||_{L^p}, \text{ for }1\leq p<\infty,
\end{align}
and
\begin{align}
 \frac{d}{dt}||u_xf||_{L^{p}}\leq  C (1+M) ||u_x f||_{L^p}+ ||f\partial_x ^2(G*P(u, u_x))||_{L^p}, \text{ for } 1\leq p<\infty.
 \end{align}
Plugging Eqs. (4.10) and (4.11) into Eqs. (4.12) and (4.13), respectively, and summing up them,  we obtain
\begin{align*}
\frac{d}{dt}\left(|| u(t)f||_{L^{p}}+||u_x (t)f||_{L^{p}}+||\rho (t)f||_{L^{p}}\right)\leq& K_1(1+M) \left(|| u_0f||_{L^{p}}+|| u_{0,x}f||_{L^{p}}+||\rho_0f||_{L^{p}}\right)\\
&+C_1\exp\left(C_2(1+M) t\right),
\end{align*}
which is taken integration and limit $N\rightarrow \infty$ to get the conclusion in the case $1\leq p <\infty$.
The constants throughout the proof are independent of $p$.
Therefore, for $p=\infty$ one can rely on the result established for the finite exponents $q$ and then let $q\rightarrow \infty$. The rest of the theorem is fully similar to that of Theorem 4.1.
\end{proof}

\end{document}